\documentclass[12pt]{amsart}
\usepackage{amsmath,amssymb,amsfonts,amsthm,amscd,amstext,amsxtra,amsopn,array,url,mathrsfs,enumerate,anysize,enumitem,mathabx}
\marginsize{2.20cm}{2.20cm}{2.20cm}{2.20cm}
\usepackage{times}
\usepackage{amsmath}
\usepackage{amssymb}
\usepackage{amsfonts}
\usepackage{xcolor}
\usepackage{hyperref}
\usepackage{amsrefs}
\newenvironment{rezabib}
  {\bibdiv\biblist\setupbib}
  {\endbiblist\endbibdiv}
  
  \def\setupbib{\catcode`@=\active}
\begingroup\lccode`~=`@
  \lowercase{\endgroup\def~}#1#{\gatherkey{#1}}
\def\gatherkey#1#2{\gatherkeyaux{#1}#2\gatherkeyaux}
\def\gatherkeyaux#1#2,#3\gatherkeyaux{\bib{#2}{#1}{#3}}

\newtheorem{theorem}{Theorem}[section]
\newtheorem{proposition}[theorem]{Proposition}
\newtheorem{lemma}[theorem]{Lemma}

\newtheorem{corollary}[theorem]{Corollary}

\numberwithin{equation}{section}

\theoremstyle{definition}

\newtheorem{remarks}[theorem]{Remarks}
\newtheorem{notation}[theorem]{Notation}

\newtheorem{remark}[theorem]{Remark}

\DeclareMathOperator{\GL}{GL}

 \DeclareMathOperator{\Ad}{Ad}
 
\usepackage{mathtools}

\renewcommand{\Re}{{\mathfrak{Re}}}
\renewcommand{\Im}{{\mathfrak{Im}}}

 \newcommand{\mymod}[1]{(\operatorname{mod} #1)}

\begin{document}

\begin{center}

\title{Fourier coefficients of  automorphic $L$-functions over primes in ray classes}
\author{Amir Akbary}
\address{Department of Mathematics and Computer Science, University of Lethbridge, Lethbridge, Alberta T1K 3M4, Canada}
\email{amir.akbary@uleth.ca}

\author[P.J. Wong]{Peng-Jie Wong}
\address{National Center for Theoretical Sciences\\
No. 1, Sec. 4, Roosevelt Rd., Taipei City, Taiwan}
\email{pengjie.wong@ncts.tw}

\thanks{Research of the first author is partially supported by NSERC. The second author is currently an NCTS postdoctoral fellow; he was supported by a PIMS postdoctoral fellowship and the University of Lethbridge during part of this research.}


\begin{abstract}
We prove Siegel-Walfisz type theorems (over long and short intervals) for the Fourier coefficients of certain  automorphic $L$-functions and Rankin-Selberg $L$-functions over number fields. 
\end{abstract}
\maketitle
\end{center}

\section{Introduction}\label{S1}

Let $p$ and $\gamma$ denote primes and positive real numbers, respectively, and  let $q>0$ and $a$ be coprime integers. The uniform version of the prime number theorem in arithmetic progressions, known as the Siegel-Walfisz theorem, provides the existence of a constant $c:=c(\gamma)>0$, depending only on $\gamma$, such that if $q\leq (\log{x})^\gamma$, then
\begin{equation}
\label{pnt}
\sum_{\substack { p\leq x\\ p\equiv a~\mymod{ {q}}}} \log{p}  = \frac{x}{\phi(q)}+O \left( x \exp \left( -c (\log{x})^{\frac{1}{2}} \right) \right),
\end{equation}
where $\phi(.)$ is Euler's totient function (see \cite[p. 133]{D}). The short interval version of this theorem states that for $y=x^\theta$, with $\theta>7/12$, one has
\begin{equation}
\label{pnt-short}
\sum_{\substack {x-y<{p}\leq x\\ {p} \equiv {a}~\mymod{{q}}}} \log{p} \sim \frac{y}{\phi(q)},
\end{equation}
as $x\rightarrow \infty$, uniformly for $q \leq (\log{x})^\gamma$ (see \cite[p. 316]{T}). 
Our goal in this paper is to prove theorems analogous to the above results for the Fourier coefficients of automorphic $L$-functions and Rankin-Selberg $L$-functions over number fields. 
Previous work on automorphic extensions of \eqref{pnt} (for example \cite{P82} and \cite{I00}) treats only the classical modular forms  or as \cite{PJ19}  is under the assumption of the Generalized Ramanujan Conjecture. Here, we prove unconditionally extensions of \eqref{pnt} and bounds of correct order of magnitude for the sum in \eqref{pnt-short},  
for certain $L$-functions of degree less than or equal to four. Our focus here is on non-abelian $L$-functions. For results related to degree one $L$-functions, see \cite{Mi56} and \cite{G}.
In order to state our results, we start with introducing some terminology and notation.

Let $F$ be a number field, and set $n_F=[F:\mathbb{Q}]$. Let $\pi$ be  an irreducible cuspidal automorphic representation of $\GL_m(\Bbb{A}_{F})$ with unitary central character.  
We shall call such representations, for short,  \emph{cuspidal representations} of $\GL_m(\Bbb{A}_{F})$. 
Associated to $\pi$, there is an integer $A_\pi\geq 1$, called the \emph{conductor} of $\pi$, and a collection of complex numbers $\alpha_\pi(j, \mathfrak{p})$,  for $1\leq j \leq m$, called the \emph{local parameters}, such that for any $\mathfrak{p} \nmid  A_\pi$,  we have $\alpha_\pi(j, \mathfrak{p})\neq 0$. We call a prime $\mathfrak{p}\nmid A_\pi$ an \emph{unramified} prime. The \emph{Generalized Ramanujan Conjecture} (GRC) states that, for $1\leq j \leq m$,  $|\alpha_\pi(j, \mathfrak{p})|= 1$ for unramified primes $\mathfrak{p}$,  and  $|\alpha_\pi(j, \mathfrak{p})|\leq1$ for ramified primes $\mathfrak{p}$. 
The truth of GRC is known for $m=1$ and for cuspidal representations that can be associated to Galois representations (for instance, the cuspidal representations arising from modular forms). Corresponding to each cuspidal representation $\pi$, there is a cuspidal representation $\check{\pi}$, the \emph{contragredient} representation. 
The collection of the local parameters for $\check{\pi}$ coincides with the collection of the complex conjugates of the local parameters for $\pi$ (i.e., $\{\alpha_{\check{\pi}} (j, \mathfrak{p})\}= \{\overline{\alpha_\pi(j, \mathfrak{p})}\}$). For integer $k\geq 1$, we set $$a_\pi(\mathfrak{p}^k)=\sum_{j=1}^{m} \alpha_\pi(j, \mathfrak{p})^k.$$ 
For a cuspidal representation $\pi$ and an id\'{e}le class character $\psi$, the \emph{twist of $\pi$ by $\psi$}, denoted $\pi \otimes \psi$, is the representation of ${\rm GL}_n(\mathbb{A}_F)$ defined by $(\pi \otimes \psi)(g)=\psi({\rm det}(g))\pi(g)$ for $g\in {\rm GL}_n(\mathbb{A}_F)$.

Let $\pi$ and $\pi^\prime$, respectively, be  cuspidal representations of $\GL_m(\Bbb{A}_{F})$ and  $\GL_{m^\prime} (\Bbb{A}_{F})$ of conductor $A_\pi$ and $A_{\pi^\prime}$.  Let $L(s, \pi \times \pi^\prime)$ denote the Rankin-Selberg $L$-function associated with $\pi$ and $\pi^\prime$, where $s=\sigma+it$ is a complex variable. 
For $\Re(s)>1$, we have
$$-\frac{L^\prime}{L}(s, \pi\times \pi^\prime)= \sum_{\mathfrak{n}\neq \mathfrak{0}}^{} \frac{\Lambda(\mathfrak{n}) a_{\pi \times \pi^\prime}(\mathfrak{n})}{{\rm N} {\mathfrak{n}^s}},
$$ 
where ${\rm N}{\mathfrak{n}}$ is the norm of the ideal $\mathfrak{n}$, $\Lambda(\mathfrak{n})$ is the number field analogue of the von Mangoldt function (i.e., $\Lambda(\mathfrak{p}^k)=\log{{\rm N} \mathfrak{p}}$ and $\Lambda(\mathfrak{n})=0$ if $\mathfrak{n}$ is not a power of a prime ideal),  and 
$$a_{\pi \times \pi^\prime}({\mathfrak{p}}^k)= a_{\pi}({\mathfrak{p}}^k) a_{\pi^\prime}({\mathfrak{p}}^k)$$
for $\mathfrak{p} \nmid (A_\pi, A_{\pi^\prime})$.
It is known that $L(s, \pi \times \pi^\prime)$ has an analytic continuation to the whole complex plane with possible simple poles at $s=i\tau$ or $s=1+i\tau$ for some $\tau\in \mathbb{R}$,  where poles exist if and only if ${\pi^\prime} \simeq  \check{\pi} \otimes |\cdot|^{-i\tau}$. In particular, $L(s, \pi\times\check{\pi})$ has only a simple pole at $s=1$. Throughout this paper we assume, unless otherwise stated, that $\pi$ and $\pi^\prime$ are normalized such that their central characters are trivial on the diagonally embedded copy of the positive reals. This normalization will ensure that the possible simple pole of $L(s, \pi \times \pi^\prime)$ at $s=1+i\tau$ can only occur at $s=1$.
(For a review of the basic properties of the Rankin-Selberg $L$-functions, see  \cite[Chapter 5]{IK}.)
We set $L(s, \pi)=L(s, \pi \times 1)$, where $1$ is the trivial representation. 

If a representation $\pi$ is isomorphic to  $\boxplus_{i=1}^{k} \pi_i$, for cuspidal representations $\pi_i$ of ${\rm GL}_{m_i}(\mathbb{A}_F)$, then $\pi$ is called an \emph{(isobaric) automorphic representation} of ${\rm GL}_{m_1+\cdots+m_k}(\mathbb{A}_F)$. For $\pi$ and $\pi^{\prime}$ as above with $1\leq m \leq 2$ and $1\leq m^\prime \leq 3$, it is known that there exist an automorphic representation of ${\rm GL}_{m m^\prime}(\mathbb{A}_F)$, denoted by $\pi \boxtimes \pi^\prime$, such that $L(s, \pi\boxtimes\pi^\prime)= L(s, \pi\times\pi^\prime)$ and $a_{\pi\boxtimes\pi^\prime}=a_{\pi \times \pi^\prime}$ (see \cite{R} and \cite{KS}).

We call a cuspidal  representation $\pi$ \emph{self-dual} if ${\pi}\simeq \check\pi$. Also, a cuspidal representation $\pi$
is called  \emph{essentially self-dual} if ${\pi}\simeq \check\pi \otimes \psi$ for some id\`ele class character $\psi$ of $F$. A cuspidal representation $\pi$ of $\GL_2(\Bbb{A}_{F})$ is called  \emph{dihedral} if it  admits a non-trivial self-twist (i.e., $\pi \simeq \pi \otimes \psi$ for some non-trivial id\`{e}le class character $\psi$ of $F$). We say that two cuspidal representations $\pi$ and $\pi^\prime$ are \emph{twist-equivalent} if $\pi^\prime \simeq \pi\otimes \psi$ for some id\`{e}le class character $\psi$ of $F$.

We are ready to state our first result. 
\begin{theorem}
\label{first}
Assume that for an automorphic representation $\Pi$, one of the following hold:

\begin{enumerate}

\item[(a)] $\Pi\simeq \pi$,  where $\pi$ is  a cuspidal representation of $\GL_m(\Bbb{A}_{F})$ for $2\leq m \leq 3$, or  $\pi$ is  a cuspidal representation of $\GL_4(\Bbb{A}_{F})$ that is not essentially self-dual.

\item[(b)] $\Pi\simeq \pi \boxtimes \check{\pi}$, where $\pi$ is a non-dihedral  cuspidal  representation of $\GL_2(\Bbb{A}_{F})$.


\item[(c)] $\Pi\simeq \pi \boxtimes {\pi}^\prime$, where $\pi$ and $\pi^\prime$ are non-dihedral  cuspidal  representations of $\GL_2(\Bbb{A}_{F})$  that are not twist-equivalent.

\end{enumerate}
Then, for any $\gamma>0$, there exists  $c:=c(\Pi)>0$ such that for any ideal $\mathfrak{q}$, with ${\rm N}{\mathfrak{q}}\leq (\log{x})^\gamma$, and any ideal $\mathfrak{a}$ relatively prime to $\mathfrak{q}$, we have
\begin{equation}
\label{twotwo}
\sideset{}{^*}\sum_{\substack {{\rm N} \mathfrak{p}\leq x\\ \mathfrak{p} \sim \mathfrak{a}~\mymod{ \mathfrak{q}}}} \Lambda(\mathfrak{p}) a_{\Pi}(\mathfrak{p}) = \frac{\delta(\Pi)}{h(\mathfrak{q})}x+O_{\Pi, \gamma} \left( x \exp \left( -c (\log{x})^{\frac{1}{2}} \right) \right),
\end{equation}
where, as later,  the superscript ``$*$"  in the above sum means that the sum is taken over primes $\mathfrak{p}$  that $\mathfrak{p}\nmid (A_\pi, A_{\pi^\prime})$,  $\mathfrak{p}\sim\mathfrak{a}~\mymod{ \mathfrak{q}}$ means that $\mathfrak{p}$ and $\mathfrak{a}$ belong to the same ray class of the ray class group modulo $\mathfrak{q}$, $h(\mathfrak{q})$ is the number of ray classes modulo $\mathfrak{q}$, $\delta(\Pi)=0$ in (a) and (c), and $\delta(\Pi)=1$ in (b).  
\end{theorem}

\begin{remarks}
{(i) In the above theorem, $\delta(\Pi)$ is the order of $- \frac{L'}{L}(s, \Pi)$ at  $s=1$.

(ii)  In the case $F=\mathbb{Q}$,  the asymptotic formula \eqref{twotwo}
becomes
\begin{equation*}
\sideset{}{^*}\sum_{\substack { p\leq x\\ p\equiv a~\mymod{ {q}}}} (\log{p}) a_{\Pi}({p}) = \frac{\delta(\Pi)}{\phi(q)}x+O_{\Pi, \gamma} \left( x \exp \left( -c (\log{x})^{\frac{1}{2}} \right) \right)
\end{equation*}
for $q\le (\log x)^{\gamma}$. In particular, if $\pi$ is a non-dihedral cuspidal  representation of ${\rm GL}_2(\mathbb{A}_\mathbb{Q})$, we have, unconditionally, 
\begin{equation*}
\sideset{}{^*}\sum_{\substack { p\leq x\\ p\equiv a~\mymod{ {q}}}} (\log{p}) |a_{\pi}({p})|^2 = \frac{1}{\phi(q)}x+O_{\pi, \gamma} \left( x \exp \left( -c (\log{x})^{\frac{1}{2}} \right) \right),
\end{equation*}
for $q\le (\log x)^{\gamma}$. This removes the assumption of the GRC in \cite[Theorem 1]{PJ19}.

(iii) An examination of the proof of Theorem \ref{first} shows that, in accordance with \eqref{pnt},  it would be possible to remove the dependence in $\gamma$ of the implied constant in the error term of \eqref{twotwo}, by making the constant $c$ in the error term to be dependent in $\gamma$. We  prefer \eqref{twotwo} as it provides an error formula
independent of $\gamma$.

(iv) The main obstacles for proving a Siegel-Walfisz type result for general automorphic $L$-functions are the lack of information on the size of their coefficients and the absence of the Siegel-type bounds for their possible exceptional zeros. The current known bounds towards the GRC together with the non-existence of exceptional zeros for degrees two and three $L$-functions, Siegel-type bounds for degree one $L$-functions,  and the  theory of Rankin-Selberg $L$-functions provide us with the needed tools in proving such a theorem for certain automorphic $L$-functions.

}
\end{remarks}

{The proof of Theorem \ref{first} is done along the classical lines by studying the analytic properties of the Rankin-Selberg $L$-functions twisted by the characters of the ray class groups. It is known that if $\pi$ is a cuspidal representation of  $\GL_m(\Bbb{A}_F)$, then $\pi \otimes \psi$ is also a cuspidal representation of $\GL_m(\Bbb{A}_F)$ for any id\'{e}le class character $\psi$ of $F$.  Thus, we can define 
\begin{equation}
\label{twist0}
L(s, \pi\times \pi^\prime \times \psi):= L(s, (\pi\otimes \psi)\times \pi^\prime).
\end{equation}
Throughout the paper, we let $L(s, \pi\times \pi^\prime \times \chi)$ be the $L$- function \eqref{twist0} attached to $\pi$, $\pi^\prime$, and 
the id\'{e}le class character associated with the ray class character $\chi$.}
Theorem \ref{first} is, in fact, a consequence of a more general theorem which we state now.

\begin{theorem}
\label{main}
Let $\pi $ and $\pi^\prime$ be cuspidal representations of $\GL_m(\Bbb{A}_{F})$ and $\GL_{m'}(\Bbb{A}_{F})$, respectively.
Let $\chi$ denote a ray class character modulo ${\mathfrak{q}}$.
Assume the following hold:

\begin{enumerate}


\item[(i)] There is $\epsilon_{0}:=\epsilon_0(\pi, \pi^\prime)>0$ such that 
\begin{equation*}
\sideset{}{^*}\sum_{x< {\rm N} \mathfrak{n}\leq x+u}  \Lambda (\mathfrak{n}) |a_{\pi\times \pi'}(\mathfrak{n})|  \ll_{\pi, \pi^\prime} u\log{x}
\end{equation*}
for $x^{1-\epsilon_0} \leq u \leq x$.


\item[(ii)] The $L$-functions  $L(s, \pi\times \pi^\prime \times \chi)$
are holomorphic everywhere except possibly  
having a simple pole at $s=1$ for exactly one ray class character $\chi=\eta$ modulo $\mathfrak{q}$.

\item[(iii)] There is a positive constant $c_{\pi, \pi^\prime}$, depending only on $\pi$ and $\pi^\prime$, such that for any ray class character $\chi$ modulo $\mathfrak{q}$, 
the $L$-function $L(s, \pi\times \pi^\prime \times \chi)$ has either no zeros or possibly only one simple real zero $\beta_\chi:=\beta(\pi, \pi^\prime, \chi)$ in the region
\begin{equation}
\label{zfree}
\sigma \geq 1-\frac{c_{\pi, \pi^\prime}}{\log{\left( ({\rm N} \mathfrak{q})
(|t|+3)\right)} }.
\end{equation}
\end{enumerate}
Then  there exists  $c:=c(\pi, \pi^\prime)>0$ such that for any ideal $\mathfrak{q}$, with ${\rm N}{\mathfrak{q}}\leq \exp((\log{x})^{1/2})$, and any ideal $\mathfrak{a}$ relatively prime to $\mathfrak{q}$, we have
 \begin{align}\label{first-2}
 \begin{split}
\sideset{}{^*}\sum_{\substack {{\rm N} \mathfrak{n}\leq x\\ \mathfrak{n} \sim \mathfrak{a}~\mymod{ \mathfrak{q}}}} \Lambda(\mathfrak{n}) a_{\pi\times \pi^\prime}(\mathfrak{n})
& = \frac{\delta(\pi, \pi^\prime, \mathfrak{q}, \mathfrak{a})}{h(\mathfrak{q})}x
 -\frac{1}{h(\mathfrak{q})} \sum_{\chi~\mymod{ \mathfrak{q}}}  \overline{\chi}(\mathfrak{a})\frac{x^{\beta_\chi}}{\beta_\chi } \\
& +O_{\pi, \pi^\prime} \left( x \exp \left( -c (\log{x})^{\frac{1}{2}} \right) \right),
 \end{split}
 \end{align}
where  the term $\frac{x^{\beta_\chi}}{\beta_\chi } $ should be omitted if $\beta_\chi$ does not exist. 
 Here $\delta(\pi, {\pi^{\prime}}, \mathfrak{q}, \mathfrak{a})=\bar{ \eta}(\mathfrak{a}) \delta(\pi\times\pi^\prime \times \eta)$, where 
$\delta(\pi\times\pi^\prime \times \eta)$ is the order of $-\frac{L'}{L}(s, \pi \times\pi^\prime \times \eta)$ at the pole $s=1$ for the possible unique ray class character $\eta$ modulo $\mathfrak{q}$ described in (ii), and $\delta(\pi, {\pi^{\prime}}, \mathfrak{q}, \mathfrak{a})=0$ otherwise.

\noindent Moreover, under the additional condition to (iii):
\begin{enumerate}
\item[(iv)] If such $\beta_{\chi}$ exists, then for any $\epsilon>0$, there is a constant $\kappa (\epsilon, \pi, \pi^\prime)$, depending on $\epsilon, \pi,$ and $\pi'$, such that
\begin{equation}
\label{s-bound}
\beta_\chi \leq 1-\frac{\kappa(\epsilon, \pi, \pi^\prime)}{{\rm N} \mathfrak{q}^\epsilon}.
\end{equation}
\end{enumerate}
Then, for any $\gamma>0$, there exists  $c:=c(\pi, \pi^\prime)>0$ such that for any ideal $\mathfrak{q}$, with ${\rm N}{\mathfrak{q}}\leq (\log{x})^\gamma$, and any ideal $\mathfrak{a}$ relatively prime to $\mathfrak{q}$, we have
\begin{equation}
\label{second}
\sideset{}{^*}\sum_{\substack {{\rm N} \mathfrak{n}\leq x\\ \mathfrak{n} \sim \mathfrak{a}~\mymod{ \mathfrak{q}}}} \Lambda(\mathfrak{n}) a_{\pi\times \pi^\prime}(\mathfrak{n}) = \frac{\delta(\pi, \pi^\prime, \mathfrak{q}, \mathfrak{a})}{h(\mathfrak{q})}x+O_{\pi, \pi^\prime, \gamma} \left( x \exp \left( -c (\log{x})^{\frac{1}{2}} \right) \right), 
\end{equation}
where $\delta(\pi, {\pi^{\prime}}, \mathfrak{q}, \mathfrak{a})$ is as defined above.
\end{theorem}

For simplicity of referring to \eqref{zfree}, throughout the paper, 
 the region given by \eqref{zfree} is called the \emph{classical zero-free region} of $L(s, \pi\times \pi^\prime \times \chi)$. Also the bound for $\beta_\chi$ given in \eqref{s-bound} is called a \emph{Siegel-type} bound for the exceptional zero $\beta_\chi$.

In this paper, we shall also prove an estimate of correct order of magnitude towards \eqref{pnt-short} for automorphic $L$-functions. In \cite{Moto},  Motohashi proved such an estimate for the sum of Hecke-Maass eigenvalues $\tau_V(p)$, associated with an irreducible representation $V$ of ${\rm PSL}_2(\mathbb{R})$, with spectral data $\nu_V$, squared over primes in short intervals.
The following is \cite[Theorem 1]{Moto}.
\begin{theorem}[{Motohashi}]
\label{Moto-t} 
There exist constants $c_0, \theta_0>0$ such that uniformly for $(\log{x})^{-1/2} \leq \theta \leq \theta_0$, $|\nu_V|^{1/\theta}\leq x$, one has
$$\sum_{x-y \leq p \leq x} \tau_V^2(p)= \frac{y}{\log{x}}\left(1+O(e^{-c_0/\theta})\right),~~y=x^{1-\theta}.$$ 
\end{theorem} 
For results similar to the above in the context of automorphic $L$-functions, see \cite{AT} and \cite{LO-T19}. The main ingredients of the proof of such results are an explicit formula similar to the classical explicit formula for the prime counting function $\psi(x)$, a classical zero-free region, and a \emph{log-free} zero-density estimate. The possibility of obtaining such estimates using a log-free zero-density estimate was first noted by Moreno \cite{Mo}.

Our next result is inspired by Theorem \ref{Moto-t}.


\begin{theorem}
\label{third}
Let $\Pi$ be as described in Theorem \ref{first} and let $\gamma, \nu>0$.
Then there are positive constants $c_0:=c_0(\Pi)$, $c:=c(\Pi)$, and $\theta_0:=\theta_0(\Pi)$ such that, for $$(\log{x})^{-1/2} \leq \theta\leq \theta_0,$$
$y=x^{1-\theta}$,
any ideal $\mathfrak{q}$ with ${\rm N}{\mathfrak{q}}\leq (\log{x})^\gamma$, and any ideal $\mathfrak{a}$ relatively prime to $\mathfrak{q}$, we have
\begin{equation}
\label{toto}
\sideset{}{^*}\sum_{\substack {x-y<{\rm N} \mathfrak{p}\leq x\\ \mathfrak{p} \sim \mathfrak{a}~\mymod{ \mathfrak{q}}}} \Lambda(\mathfrak{p}) a_{\Pi}(\mathfrak{p})
 = y \left(\frac{\delta(\Pi)}{h(\mathfrak{q})}
 +O_{\Pi, \gamma, \nu} \left( \exp \left( -c (\log{x})^{1-\nu} \right)\right)
 +O_{\Pi} \left(e^{- {c_0}/{\theta}} \right) \right),
\end{equation}
where $\delta(\Pi)$ is as defined in Theorem \ref{first}.

\end{theorem}

The above theorem is also a consequence of the following more general assertion.

\begin{theorem}
\label{main2}
Let $\pi $ and $\pi^\prime$ be cuspidal representations of $\GL_m(\Bbb{A}_{F})$ and $\GL_{m'}(\Bbb{A}_{F})$, respectively.
Assume the following hold:

\begin{enumerate}


 \item[(i)] There is $\epsilon_{0}:=\epsilon_0(\pi, \pi^\prime)>0$ such that 
\begin{equation*}
\sideset{}{^*}\sum_{x< {\rm N} \mathfrak{n}\leq x+u}  \Lambda (\mathfrak{n}) |a_{\pi\times \pi'}(\mathfrak{n})|  \ll_{\pi, \pi^\prime} u\log{x}
\end{equation*}
for $x^{1-\epsilon_0} \leq u \leq x$.

\item[(ii)] The $L$-functions  $L(s, \pi\times \pi^\prime \times \chi)$ 
are holomorphic everywhere except possibly  
having a simple pole at $s=1$ for exactly one ray class character $\chi=\eta$ modulo $\mathfrak{q}$.

\item[(iii)] There is a positive constant $c_{\pi, \pi^\prime}$, depending only on $\pi$ and $\pi^\prime$, such that for any ray class character $\chi$  modulo $\mathfrak{q}$, 
the $L$-function $L(s, \pi\times \pi^\prime \times \chi)$ has either no zeros or possibly only one simple real zero $\beta_\chi:=\beta(\pi, \pi^\prime, \chi)$ in the region
$$\sigma \geq 1-\frac{c_{\pi, \pi^\prime}}{\log{\left( ({\rm N}\mathfrak{q})
(|t|+3)\right)} }.$$

\item[(iv)] There is a positive constant $d_{\pi, \pi^\prime}$ such that for $T\geq 1$ and $0\leq \sigma\leq 1$, we have
$$N(\sigma, T, \pi \times \pi^\prime \times \chi) \ll_{\pi, \pi^\prime} \left(({\rm N} \mathfrak{q})T\right)^{d_{\pi,\pi^\prime}(1-\sigma)},$$
where
$$
N(\sigma,T, \pi\times\pi^\prime\times \chi) =\#\{\rho=\Re(\rho)+i \Im(\rho) \mid  L(\rho,\pi \times \pi^\prime\times \chi)=0,~ \Re(\rho)\ge \sigma,~ |\Im(\rho)|\le T \}. 
$$
\end{enumerate}
Then there exists a positive constant $c_0:=c_0(\pi, \pi^\prime)$ such that, for $$(\log{x})^{-1/2} \leq \theta\leq \min\left\{\frac{1}{10d_{\pi, \pi^\prime}}, \frac{\epsilon_0}{4}\right\}$$ and $y=x^{1-\theta}$, we have

%

\begin{equation}
\label{fourfour}
\sideset{}{^*}\sum_{\substack {x-y<{\rm N} \mathfrak{n}\leq x\\ \mathfrak{n} \sim \mathfrak{a}~\mymod{ \mathfrak{q}}}} \Lambda(\mathfrak{n}) a_{\pi\times \pi^\prime}(\mathfrak{n}) = y \left(\frac{\delta(\pi, \pi^\prime, \mathfrak{q}, \mathfrak{a})}{h(\mathfrak{q})}+O\left( \frac{1}{h(\mathfrak{q})}     \sum_{\chi~\mymod{ \mathfrak{q}}}  x^{\beta_\chi -1} \right)+O_{\pi, \pi^\prime} \left(e^{- {c_0}/{\theta}} \right) \right),
\end{equation}
uniformly for all $\mathfrak{q}$, with ${\rm N} \mathfrak{q} \leq x^\theta$, where $\delta(\pi, \pi^\prime, \mathfrak{q}, \mathfrak{a})$ is as defined in Theorem \ref{main}.

\noindent Moreover, under the additional condition to (iii):
\begin{enumerate}
\item[(v)]
If  the possible exceptional zero $\beta_{\chi}$ of $L(s, \pi\times \pi^\prime \times \chi)$  exists, then for any $\epsilon>0$, there is a constant $\kappa (\epsilon, \pi, \pi^\prime)$, depending on $\epsilon, \pi,$ and $\pi'$, such that
$$\beta_\chi \leq 1-\frac{\kappa(\epsilon, \pi, \pi^\prime)}{{\rm N} \mathfrak{q}^\epsilon}.$$
\end{enumerate}
Then, for any $\gamma, \nu>0$, there exists  $c:=c(\pi, \pi^\prime)>0$ such that
\begin{equation}
\label{last-main-est}
\sideset{}{^*}\sum_{\substack {x-y<{\rm N} \mathfrak{n}\leq x\\ \mathfrak{n} \sim \mathfrak{a}~\mymod{ \mathfrak{q}}}} \Lambda(\mathfrak{n}) a_{\pi\times \pi^\prime}(\mathfrak{n})
 = y \left(\frac{\delta(\pi, \pi^\prime, \mathfrak{q}, \mathfrak{a})}{h(\mathfrak{q})}
 +O_{\pi, \pi', \gamma, \nu} \left( \exp \left( -c (\log{x})^{1-\nu} \right)\right)
 +O_{\pi, \pi'} \left(e^{- {c_0}/{\theta}} \right) \right),
\end{equation}
uniformly for all $\mathfrak{q}$, with $ {\rm N} \mathfrak{q} \leq (\log x )^\gamma$, where $\delta(\pi, \pi^\prime, \mathfrak{q}, \mathfrak{a})$ is as defined in Theorem \ref{main}.

\end{theorem}

\begin{remark}
We note that by the work of Soundararajan and Thorner \cite[Theorem 2.4 and Corollary 2.6]{ST}, the condition (i) of Theorem \ref{main} and conditions (i) and (iv) of Theorem \ref{main2} hold, unconditionally, whenever $F=\Bbb{Q}$. Also, recently, Humphries and Thorner \cite[Theorem 2.4]{HT} showed that the classical zero-free region (the condition (iii) in Theorems \ref{main} and \ref{main2}) and the Siegel-type bound (the condition (iv) of Theorem \ref{main} and the condition (v) of Theorem \ref{main2}) are valid if $\pi'=\check{\pi}$, $(q, A_\pi)=1$,  and $F=\Bbb{Q}$. Thus, by Theorem \ref{main}, for any cuspidal representation $\pi$ of $\GL_m(\Bbb{A}_{\Bbb{Q}})$ and $\gamma>0$, there exists $c:=c(\pi)>0$ such that
\begin{equation*}
\sideset{}{^*}\sum_{\substack { p^k\leq x\\ p^k\equiv a~\mymod{ {q}}}} (\log{p}) |a_{\pi}({p^k})|^2 
= \frac{1}{\phi(q)}x+O_{\pi, \gamma} \left( x \exp \left( -c (\log{x})^{\frac{1}{2}} \right) \right)
\end{equation*}
for $q\le (\log x)^{\gamma}$ with $(q, A_\pi)=1$. Moreover, it follows from Theorem \ref{main2} that there are positive constants $c_0:=c_0(\pi)$, $c:=c(\pi)$, and $\theta_0:=\theta_0(\pi)$ such that, for $(\log{x})^{-1/2} \leq \theta\leq \theta_0$,  $y=x^{1-\theta}$, $\gamma, \nu>0$,  and any $q\leq (\log{x})^\gamma$ with $(q, A_\pi)=1$,  we have
\begin{equation*}
\sideset{}{^*}\sum_{\substack {x-y< p^k \leq x\\ p^k\equiv a~\mymod{ {q}} }} (\log{p}) |a_{\pi}({p^k})|^2 
 = y \left(\frac{1}{\phi(q)}
 +O_{\pi, \gamma, \nu} \left( \exp \left( -c (\log{x})^{1-\nu} \right)\right)
 +O_{\pi} \left(e^{- {c_0}/{\theta}} \right) \right).
\end{equation*}
In addition, by Proposition \ref{Wu-Ye}, for $1\leq m \leq 4$, the same estimates hold for the corresponding sums \emph{only} supported over primes.
\end{remark}

In the rest of the paper, we prove Theorems \ref{first} and \ref{third}. The structure of the paper is as follows. In Section \ref{two}, we start by reviewing some facts and results from the theory of automorphic forms that will be used in the proofs of our main assertions. In Section \ref{Section3} we show that Theorem \ref{first} is a consequence of Theorem \ref{main}  and then in Section \ref{four} we give a proof of Theorem \ref{main}.  Similarly we describe in Section \ref{five} that Theorem \ref{main2} implies Theorem \ref{third} and then in Section \ref{six} we prove Theorem \ref{main2}.
\begin{notation}
{Throughout the paper $F$ is a number field of degree $n_F$, $\mathfrak{p}$ is a prime ideal of $F$, ${\rm N}\mathfrak{q}$ is the norm of an ideal $\mathfrak{q}$ of $F$, $\pi$, $\pi^\prime$ are cuspidal representations, $A_\pi$ and $\mathfrak{q}(\pi)$ are respectively the conductor and the  analytic conductor of $\pi$, $L(s, \pi)$ is the automorphic $L$-function associated with $\pi$, $\pi \otimes \psi$ is the twist of $\pi$ by an id\'{e}le class character $\psi$, $\pi\boxplus \pi^\prime$ is the isobaric sum of two cuspidal representations, $\pi\boxtimes\pi^\prime$ is the automorphic representation associated to two cuspidal representations if it exists, $L(s, \pi \times \pi^\prime)$ is the Rankin-Selberg $L$-function associated with $\pi$ and $\pi^\prime$, $A_{\pi\times \pi^\prime}$ and $\mathfrak{q}(\pi\times \pi^\prime)$ are respectively  the conductor and the analytic conductor of $\pi\times \pi^\prime$, $\Lambda(\cdot)$ is the number field von Mangoldt function, 
$h(\mathfrak{q})$ is the number of ray classes modulo $\mathfrak{q}$,   $\chi$ and $\eta$ are ray class characters modulo $\mathfrak{q}$, and $\psi$ 
is an id\'{e}le class character of ${F}$.
We use the Landau big-$O$ and Vinogradov $\ll$ notations with their usual meanings. The dependence of the implied constants on the parameter $t$ is denoted by $O_t(\cdot)$ or $\ll_t$. 
Throughout the paper we have suppressed the dependence of the constants to the base field $F$.}
\end{notation}
\subsection*{Acknowledgement}
The authors would like to thank Jesse Thorner for the valuable comments and suggestions in an earlier version of this paper.

\section{Preliminaries}
\label{two}

In this section,  we review some results from the theory of automorphic representations which we will need in the proof of our theorems. Let $\pi$ be  a cuspidal representation of $\GL_m(\Bbb{A}_{F})$ with local parameters $\alpha_\pi(j, \mathfrak{p})$.
%
%
For cuspidal representations $\pi$ and $\pi^\prime$ and integer $k\geq 1$, we set $$a_{\pi \times \pi^\prime}({\mathfrak{p}}^k)=\sum_{i=1}^{m} \sum_{j=1}^{m^\prime} \alpha_{\pi \times \pi^\prime}(i, j, \mathfrak{p})^k.$$
Here, $\{\alpha_{\pi \times \pi^\prime}(i, j, \mathfrak{p}) \mid  1\leq i \leq m~{\rm and}~ 1\leq j\leq m^\prime\}$ is the collection of local parameters of $L(s, \pi \times \pi^\prime)$ at the prime  $\mathfrak{p}$. For $\mathfrak{p} \nmid (A_\pi, A_{\pi^\prime})$, $1\leq i \leq m$ and $1\leq j\leq  m^\prime$, we have 
$$\alpha_{\pi \times \pi^\prime}(i, j, \mathfrak{p})= \alpha_{\pi}(i, \mathfrak{p}) \alpha_{\pi^\prime}(j, \mathfrak{p}).$$
The best general known bound for $a_{\pi \times \pi^\prime} (\mathfrak{p}^k)$
is 
\begin{equation}
\label{bound-general}
|a_{\pi \times \pi^\prime} (\mathfrak{p}^k)| \leq m m^\prime ({\rm N}{\mathfrak{p}^k)}^{1-\frac{1}{m^2+1}-\frac{1}{(m^\prime)^2+1}}, 
\end{equation}
see \cite[Formula (4)]{B}.
We denote by $A_{\pi \times \pi^\prime}$  the  \emph{conductor} of $\pi \times \pi^\prime$. 
The \emph{analytic conductor} of $\pi \times \pi^\prime$ is $$\mathfrak{q}(\pi\times \pi^\prime) =A_{\pi \times \pi^\prime} \prod_{i=1}^{m}\prod_{j}^{m^\prime} \prod_{v\in S_\infty} (|\kappa_{\pi\times \pi^\prime}(i, j, v)| +3),$$ where $S_\infty$ is the collection of infinite places of $F$ and the $\kappa_{\pi\times \pi^\prime}(i, j, v)$'s are the parameters of $L(s, \pi \times \pi^\prime)$ at infinity. We set $\mathfrak{q}(\pi):= \mathfrak{q}(\pi \times 1)$. It is shown in \cite[Lemma A.2]{HB} that 
\begin{equation}
\label{Bushnell-2}
\mathfrak{q}(\pi \times \pi^\prime) \leq C_0^{m+m^\prime} \mathfrak{q}(\pi)^{m^\prime} \mathfrak{q}(\pi^\prime)^m,
\end{equation}
for an absolute constant $C_0>0$. The parameters at infinity  satisfy the bound
\begin{equation}
\label{infinite}
|\Re(\kappa_{\pi\times\pi^\prime}(i, j, v)|\leq 1-\frac{1}{m^2+1}-\frac{1}{(m^\prime)^2+1}
\end{equation}
analogous to \eqref{bound-general} (see \cite[Formula (6)]{B}). Also, from the functional equation of the ray class $L$-function $L(s, \chi)$, we know that the parameters at infinity of $\chi$ are either zero or one (see \cite[p. 129]{IK}). Thus, following an analysis of the parameters at infinity of Rankin-Selberg $L$-functions, as done in \cite[pp. 1119-1121]{HB},  we have that 
\begin{equation}
\label{infinity-twist}
\kappa_{\pi\times\pi^\prime\times \chi}(i,j, v)= \kappa_{\pi\times \pi^\prime}(i,j, v)+c,
\end{equation}
where $c$ belongs to a finite set depending only on $\pi$ and $\pi^\prime$.

The following estimate on the average size of $a_{\pi \times \pi^\prime} (\mathfrak{n})$ over a short interval will play an important role in the proofs of Theorems \ref{first} and \ref{third}.

\begin{proposition}
\label{Landau}
Let $\pi$ and $\pi^\prime$ be cuspidal  representations of ${\rm GL}_m (\mathbb{A}_F)$ and ${\rm GL}_{m^\prime} (\mathbb{A}_F)$, respectively. 
Assume that there is $\epsilon_{\pi, \pi^\prime}>0$ such that 
\begin{equation}
\label{primepowers'}
\sum_{r\geq 2}\sideset{}{^*}\sum_{ ~{\rm N} \mathfrak{p}^r\leq x}  \Lambda (\mathfrak{p}^r) |a_{\pi\times \pi'}(\mathfrak{p}^r)|  \ll_{\pi, \pi^\prime} x^{1-\epsilon_{\pi, \pi^\prime}}.
\end{equation}
Suppose that $x^{1-{\epsilon}} \leq u\leq x$,  where  
$$0<\epsilon <\min\left\{ \frac{1}{n_F\max\{m, m^\prime\}^2+1}, 
 \epsilon_{\pi, \pi^\prime} \right\}.$$
Then 
$$\sideset{}{^*}\sum_{x< {\rm N} \mathfrak{n}\leq x+u}  \Lambda (\mathfrak{n}) |a_{\pi\times \pi'}(\mathfrak{n})|  \ll_{\pi, \pi^\prime, \epsilon} u \log{x}.
$$
\end{proposition}
\begin{proof}
Let $b_{\pi \times \pi^\prime}(\mathfrak{n})$ be the coefficients coming from the formal identity 
$$\sum_{ \mathfrak{n}\neq \mathfrak{0}}^{} \frac{b_{\pi \times \pi^\prime}(\mathfrak{n})}{{\rm N} \mathfrak{n}^s}
=\prod_{\mathfrak{p}} \prod_{i=1}^{m} \prod_{j=1}^{m^\prime} \left(1- \frac{\alpha_{\pi \times \pi^\prime}(i, j, \mathfrak{p})}{{\rm N} \mathfrak{p}^s}  \right)^{-1}.
$$
From the number field analogue of \cite[Eq. (1.10)]{Mi06}, we have
\begin{equation}
\label{L-equation}
\sum_{{\rm N} \mathfrak{n}\leq x }  b_{\pi\times \check{\pi}}(\mathfrak{n})
 = c_{\pi} x+O_{\pi, \lambda} \left( x^{\frac{n_F m^2-1}{n_F m^2+1} +\lambda} \right),
\end{equation}
for some $c_{\pi}>0$ and any $\lambda>0$.
We note that by \cite[Lemma a]{HR}, each $b_{\pi\times \check{\pi}}(\mathfrak{n})$ is non-negative, also $b_{\pi \times \pi^\prime} (\mathfrak{p})= a_{\pi}(\mathfrak{p})a_{ \pi'}(\mathfrak{p}) $ for $\mathfrak{p}\nmid (A_\pi, A_\pi^\prime)$. 
Thus, by employing \eqref{primepowers'}, the Cauchy-Schwarz inequality, and \eqref{L-equation}, we have
\begin{align*}
\sideset{}{^*}\sum_{x< {\rm N} \mathfrak{p}^m\leq x+u}  (\log{{\rm N} \mathfrak{p}}) |a_{\pi\times \pi'}(\mathfrak{p}^m) |
 &\ll (\log{x})\sideset{}{^*}\sum_{x< {\rm N} \mathfrak{p}\leq x+u }  |a_{\pi}(\mathfrak{p})a_{ \pi'}(\mathfrak{p})| + x^{1-\epsilon_{\pi, \pi^\prime}}\\
&\ll (\log{x}) \left(  \sum_{x< {\rm N} \mathfrak{n}\leq x+u }  b_{\pi\times \check{\pi}}(\mathfrak{n})\right)^{\frac{1}{2}} \left( \sum_{x<{\rm N} \mathfrak{n}\leq x+u }  b_{\pi'\times \check{\pi}'}(\mathfrak{n})\right)^{\frac{1}{2}}+x^{1-\epsilon_{\pi, \pi^\prime}}\\
&\ll  (\log{x}) u^{\frac{1}{2}} u^{\frac{1}{2}}+x^{1-\epsilon_{\pi, \pi^\prime}}\\
&\ll_{\pi,\pi', \epsilon} u \log{x}, 
\end{align*}
as long as $x^{1-{\epsilon}}\leq u \leq x$.
\end{proof}

We next note that the truth of \eqref{primepowers'} is known for some small $m$ and $m^\prime$. The following is a number field adaptation of \cite[Lemma 3.1]{WY07} (combined with the  Cauchy-Schwarz inequality).

\begin{proposition}
\label{Wu-Ye}
The assertion \eqref{primepowers'} holds for $m, m^\prime \in \{1, 2, 3, 4 \}$.
\end{proposition}

We continue with an important theorem on the automorphy of $L(s, \pi \times \pi^\prime)$ for ${\rm GL}_2$ representations.

\begin{theorem}
\label{GJ}
%
%
\noindent (i) Let $\pi$ and $\pi^\prime$ be cuspidal representations of $\GL_2(\mathbb{A}_F)$. Then there is an automorphic representation $\pi\boxtimes \pi^\prime$ of $\GL_4(\mathbb{A}_F)$ for which $L(s, \pi\boxtimes \pi^\prime)= L(s, \pi \times \pi^\prime)$. In addition, if $\pi$ and $\pi^\prime$ are non-dihedral, then $\pi\boxtimes \pi^\prime$ is cuspidal whenever  $\pi$ and $\pi^\prime$ are not twist-equivalent.

(ii) Let $\pi$ be a non-dihedral cuspidal representation of $\GL_2(\mathbb{A}_F)$.
Then there is a cuspidal  representation $\Ad(\pi)$ of $\GL_3(\mathbb{A}_F)$ such that
$$L(s, \pi\times\check{\pi}\times \psi)=L(s, \Ad(\pi) \otimes \psi) L(s, \psi)$$
for any id\'{e}le class character $\psi$ of $F$.
%

\end{theorem}

\begin{proof}
Part (i) is contained in Theorem M of \cite{R}. 
Part (ii) is a consequence of \cite[Theorem 9.3]{GJ}.
\end{proof}

We next state several results on the existence of the classical zero-free region for $L(s, \pi\times \pi^\prime)$.
The following theorem is Theorem A.1 of \cite{HB}.

\begin{theorem}
\label{zero-free}
Let $\pi$ and $\pi^\prime$ be cuspidal representations of ${\rm GL}_m(\mathbb{A}_F)$ and $ {\rm GL}_{m^\prime}(\mathbb{A}_F)$, respectively. Assume as usual that both $\pi$ and $\pi^\prime$ are normalized such that their central characters are trivial on the diagonally embedded copy of the positive reals. Assume that $\pi^\prime$ is self-dual.  Then there is an effective absolute constant $c>0$ such that $L(s, \pi\times \pi^\prime)$ is not vanishing for all $s=\sigma+it \in \mathbb{C}$ satisfying
$$\sigma\geq 1-\frac{c}{(m+m^\prime)^3 \log\left(\mathfrak{q}(\pi) \mathfrak{q}(\pi^\prime) (|t|+3)^{mn_F} \right)},$$
with the possible exception of one real zero whenever $\pi$ is also self-dual.
\end{theorem}
\begin{proof}
The proof is given in \cite[Appendix A]{HB}.
\end{proof}

\begin{corollary}
\label{pi}
If $\pi$ is a cuspidal representation of ${\rm GL}_m(\mathbb{A}_F)$, then $L(s, \pi\times \chi)$ satisfies a classical zero-free region for any given ray class character $\chi$ modulo $\mathfrak{q}$. Moreover, if $\pi\otimes\chi$ is not self-dual, then $L(s, \pi \times \chi)$ admits no exceptional zero.
\end{corollary}
\begin{proof}
Recall that $\pi$ is normalized such that its central character $\omega_\pi$ is trivial on the diagonally embedded copy of positive reals. Note that $\pi\otimes \chi$ is a cuspidal representation with the central character $\omega_{\pi \otimes \chi}$, where
\begin{equation}
\label{central-twist}
\omega_{\pi \otimes \chi} = \omega_\pi \chi^{m}.
\end{equation}
{Now, since $\chi$ is trivial on the diagonally embedded copy of positive reals}, then  \eqref{central-twist} shows that $\pi \otimes \chi$ is normalized.
Thus, by replacing $\pi$ in Theorem \ref{zero-free} with $\pi\otimes\chi$ and $\pi^\prime$ with $1$ and noting that, by \eqref{Bushnell-2}, 
\begin{equation}
\label{twist}
\mathfrak{q}(\pi \times \chi) \ll_\pi  {\rm N}(\mathfrak{q})^m,
\end{equation} 
we will have the desired result.
\end{proof}
\begin{remark}
We note that the $L$-functions $L(s, \pi\otimes\chi)$, $L(s, \pi\times\chi)$, and $L(s, \pi\times 1\times \chi)$ are the same. We shall use this fact throughout our discussion.
\end{remark}

We also need a result on the classical zero-free region for $L(s, \pi)$ when $\pi$ is not necessarily normalized.

\begin{proposition}
\label{zero-free2}
Let $\pi$ be a cuspidal representation (not necessarily normalized) of ${\rm GL}_m(\mathbb{A}_F)$. Assume, further, that $L(s, \pi \times \pi)$ is entire if $\pi \not\simeq \check{\pi}$. Then there is an effective absolute constant $c>0$ such that $L(s, \pi)$ is non-vanishing for all $s=\sigma+it\in \mathbb{C}$ satisfying 
$$\sigma \geq 1- \frac{c}{(m+1)^3 \log(\mathfrak{q}(\pi) (|t|+3)^{m n_F})},$$
with the possible exception of one real zero whenever $\pi$ is self-dual. 
\end{proposition}
\begin{proof}
The proof is the same as the proof of Theorem A.1 in \cite{HB}, when $\pi^\prime=1$. The main facts used in the proof are that $L(s, \pi\times\pi)$ is entire if $\pi \not\simeq \check{\pi}$ and $L(s, \pi\times \check{\pi})$ has a simple pole at $s=1$ and it is holomorphic everywhere else.
See also \cite[Theorem 5.10]{IK}.
\end{proof}

\begin{corollary}
\label{self-dual2}
Let $\pi$ be a self-dual cuspidal representation (not necessarily normalized) of ${\rm GL}_m(\mathbb{A}_F)$. Then $L(s, \pi\times\chi)$ satisfies a classical zero-free region for any given ray class character $\chi$  modulo $\mathfrak{q}$.
\end{corollary}
\begin{proof}
We consider two cases.

Case 1: Assume that $\pi\otimes \chi$ is self-dual. Then 
by Proposition \ref{zero-free2} a classical zero-free region is furnished for $L(s, \pi \times \chi)$. 

Case 2: Assume that $\pi\otimes \chi$ is not self-dual. We claim that in such case $\pi\otimes \chi \not\simeq {(\pi \otimes \chi)}^{\vee} \otimes  | \cdot |^{i\tau}$ for any $\tau \in \mathbb{R}$. Suppose, on the contrary, that
$\pi\otimes \chi \simeq {(\pi \otimes \chi)}^{\vee} \otimes  |\cdot|^{i\tau}$ for some $\tau \in \mathbb{R}$. As $\pi$ is self-dual, we conclude that
$$\pi\otimes \chi^2 \simeq \pi \otimes |\cdot|^{i\tau}.$$
From here, by employing \eqref{central-twist}, we have
$$\omega_\pi  \chi^{2m}=\omega_\pi |\cdot|^{im\tau}.$$
As $\pi$ is unitary, the last identity implies that 
$$\chi^{2m}= |\cdot|^{im\tau}.$$
Now, since $\chi$ is of finite order, we conclude that $\tau=0$ and thus $\pi\otimes \chi$ is self-dual, a contradiction. 

Thus, $\pi\otimes \chi \not\simeq {(\pi \otimes \chi)}^{\vee} \otimes  |\cdot|^{i\tau}$ for any $\tau \in \mathbb{R}$, which implies the
holomorphy of $L(s, (\pi \otimes \chi) \times (\pi\otimes \chi))$ everywhere. 
Therefore, by Proposition \ref{zero-free2} and \eqref{twist}, $L(s, \pi \times \chi)$ has a classical zero-free region. 
\end{proof}

The following result shows that, for certain ${\rm GL}_2$ representations, we can dispensed with the self-duality condition of $\pi^\prime$ in Theorem \ref{zero-free}. 
\begin{theorem}
\label{zero-free3}
Let $\pi$ and $\pi^\prime$ be non-dihedral cuspidal representations of ${\rm GL}_2(\mathbb{A}_F)$. Assume that $\pi^\prime$ is not twist-equivalent to $\pi$.  Then there is an effective absolute constant $c>0$ such that $L(s, \pi \times \pi^\prime)$ has no zero in the region 
 $$\sigma \geq 1- \frac{c}{\log(A_{\pi \times \pi^\prime} (|t|+2+\lambda)^{{4} n_F})},$$
where
$\lambda$ is the maximum of the absolute value of the infinite parameters of $\pi$ and $\pi^\prime$.
\end{theorem}
\begin{proof}
This is Theorem 4.12(b) in \cite{RW03}.
\end{proof}

\begin{remark}
In Theorem 4.12(b) in \cite{RW03}, it is further assumed that $\pi$ and $\check{\pi}^{\prime}$ are not twist-equivalent by a product of a quadratic character and $|\cdot|^{i\tau}$. However, since $\pi^\prime$ is a cuspidal representation of ${\rm GL}_2 (\mathbb{A}_F)$, $\check{\pi}^\prime\simeq \pi^\prime \otimes \omega_{\pi^\prime}^{-1}$, where $\omega_{\pi^\prime}$ is the central character of $\pi^\prime$. Thus, if $\pi^\prime$ is not twist-equivalent to $\pi$, then $\pi$ and $\check{\pi}^{\prime}$ are not twist-equivalent by a product of a quadratic character and $|\cdot|^{i\tau}$.
\end{remark}
\begin{corollary}
\label{gl2}
Let $\pi$ and $\pi^\prime$ be non-dihedral cuspidal representations of ${\rm GL}_2(\mathbb{A}_F)$. Assume that $\pi^\prime$ is not twist-equivalent to $\pi$.
Then $L(s, \pi\times\pi^\prime\times\chi)$ satisfies a classical zero-free region for any given ray class character $\chi$  modulo $\mathfrak{q}$.
\end{corollary}

\begin{proof}
Under the given assumptions, $\pi\otimes\chi$ and $\pi^\prime$ are non-dihedral, and moreover $\pi^\prime$ is not twist-equivalent to $\pi\otimes \chi$. In addition, from the theory of Rankin-Selberg $L$-functions (see, e.g., \cite[p. 97]{IK} for $F=\mathbb{Q}$), we have the relation
{$$|\kappa_{\pi\otimes \chi}(i, j, v)|\leq |\kappa_{\pi}(i, v)|+|\kappa_\chi(j, v)|\leq |\kappa_{\pi}(i, v)|+1$$}between
the infinite parameters of $\pi \otimes \chi$ and the infinite parameters of $\pi$. Thus the claimed assertion is a direct corollary of Theorem \ref{zero-free3} and \eqref{Bushnell-2}.
\end{proof}

We next review some results on the existence and the locations of the exceptional zeros of $L$-functions. We start by a Siegel-type bound on the location of the exceptional zeros of the ray class $L$-functions.

\begin{theorem}
\label{Grossen}
Let $\chi$ be a ray class character modulo $\mathfrak{q}$. Let $\beta_{\chi}$ be the possible exceptional zero of $L(s, \chi)$. Then, given $\epsilon>0$,  
there is a constant $\kappa (\epsilon)$, depending only on $\epsilon$, such that
$$
\beta_{\chi} \leq 1-\frac{\kappa(\epsilon)}{{\rm N} \mathfrak{q}^\epsilon}.
$$
\end{theorem}
\begin{proof}
See \cite[Section 1, Lemma 11]{Mi56}.
\end{proof}

The following result summarizes some cases for which the non-existence of exceptional zeros is known.

\begin{theorem}
\label{Siegelzero}
(i) Let $\pi$ be a cuspidal representation of $\GL_n(\mathbb{A}_F)$, and  assume that either $\pi$ is not self-dual or $n=2, 3$. Then 
$L(s, \pi)$ does not admit an exceptional zero.

\noindent (ii) Let $\pi$ and $\pi'$ be non-dihedral cuspidal representations of $\GL_2(\mathbb{A}_F)$  that are not twist-equivalent. Then $L(s, \pi \times \pi^\prime)$ admits no exceptional zero.

%
\end{theorem}

\begin{proof}
Part (i) is a consequence of  \cite[Corollary 3.2]{HR},  \cite[Theorem C(3)]{HR}, and  \cite[Theorem 1]{Ba97}.
 Part (ii) follows from  \cite[Theorem A]{RW03}.
\end{proof}


Finally, we deduce a log-free zero-density estimate for an automorphic representation twisted by a ray class character.

\begin{theorem}\label{log-free}
Let $F$ be a number filed of degree $n_F$. 
Let $\Pi = \boxplus_i \pi_i$ be an automorphic representation for   $\GL_m(\Bbb{A}_{F})$, where each $\pi_i$ is a cuspidal representation for   $\GL_{m_i}(\Bbb{A}_{F})$. Set 
$$
N(\sigma,T, \Pi) =\#\{\rho=\Re(\rho) + i\Im(\rho) \mid  L(\rho,\Pi)=0, ~\Re(\rho)\ge \sigma,~ |\Im(\rho)|\le T \}. 
$$
Then there is an absolute constant $c_1>0$ such that for $T\ge 1$ and $0\leq \sigma\leq 1$, one has
$$
N(\sigma,T, \Pi)\ll_\Pi  \sum_i m_i^2(\mathfrak{q}(\pi_i) T^{n_F})^{ c_1 m_i^2(1-\sigma)}.
$$
Consequently, given a ray class character $\chi$ modulo $\mathfrak{q}$, for $T\ge 1$ and $0\leq \sigma \leq 1$, there is a positive constant $d_{\Pi}$ such that
$$N(\sigma, T, \Pi \times \chi) \ll_{\Pi} \left(({\rm N} \mathfrak{q})T\right)^{d_{\Pi}(1-\sigma)}.$$

\end{theorem}

\begin{proof}
For each $i$,  \cite[Corollary  1.2]{LO-T19} asserts that
$$
N(\sigma,T, \pi_i)
\ll m_i^2(\mathfrak{q}(\pi_i) T^{n_F})^{ c_1 m_i^2(1-\sigma)}
$$
for $1/2 \leq \sigma \leq 1$ and $T\geq n_F$. Now the first part of the theorem follows immediately from the fact that  $N(\sigma,T, \Pi) = \sum_i N(\sigma,T, \pi_i) $. Finally, by employing the bound \eqref{twist} for 
$\mathfrak{q}(\pi_i \times \chi)$,
we conclude the proof. (Note that the above bound trivially extends to $0\leq \sigma \leq 1$.)
\end{proof}

\section{Theorem \ref{main} implies Theorem \ref{first} }
\label{Section3}

We need to show that the conditions (i), (ii), (iii), and (iv) of Theorem \ref{main} hold for pairs $\pi$ and $\pi^\prime$ associated with $\Pi$ satisfying either  (a), (b), or (c). Note that $\pi^\prime=1$ in (a) and $\pi^\prime=\check{\pi}$ in (b). 
We observe that, by Propositions \ref{Landau} and \ref{Wu-Ye},  (i) holds for $\pi$ and $\pi^\prime$ associated with $\Pi$ in (a), (b), or (c). We now establish (ii), (iii), and (iv),  for corresponding $\pi$ and $\pi^\prime$ in (a), (b), or (c).

(a)
The condition (ii) is true, since for any character $\chi$,
$\pi \otimes \chi$ is  a cuspidal representation of ${\rm GL}_m(\mathbb{A}_F)$, with $m>1$,  and thus $L(s, \pi\times 1\times \chi)$ is holomorphic. Also, by Corollary \ref{pi}, (iii) holds. 

To verify the condition (iv), we first note that, by Theorem \ref{Siegelzero}(i),  for $m=2$ and $3$
 none of the $L(s,\pi\times1\times\chi)$'s admit an exceptional zero in their classical zero-free region. 
If  $\pi$ is not essentially self-dual, then each $\pi\otimes \chi$ is not self-dual. (Suppose, on the contrary, that for some character $\chi$ the contragredient representation of $\pi\otimes \chi$ is equivalent to $\pi \otimes \chi$.
A direct calculation shows that ${\pi} \simeq \check{\pi} \otimes\bar{\chi}^2 $, a contradiction.) Therefore, by Corollary \ref{pi}, $L(s, \pi\times 1\times \chi)$ admits no exceptional zero if  $\pi$ is not essentially self-dual.
Thus,  (iv) holds trivially.


(b) 
%
Let $\Pi\simeq \pi\boxtimes  {\check{\pi}}$. Then, 
by  Theorem \ref{GJ}(ii) we have   
\begin{equation}\label{factor22}
L(s, \pi\times  \check{\pi}\times\chi) =L(s,\Ad(\pi)\otimes\chi)L(s,\chi)
\end{equation}
for any ray class character $\chi$.
Since $\Ad(\pi)$ is cuspidal,  $\Ad(\pi)\otimes\chi$ is also cuspidal, and so $L(s,\Ad(\pi)\otimes\chi)$ is holomorphic. Hence, if $L(s,\pi\times  \check{\pi} \times\chi)$ admits a pole, then it is contributed by $L(s, \chi)$. This happens only if $\chi$ is the principal character $\chi_0$.  
Thus (ii) holds.

To verify (iii), we note that since $\Ad(\pi)$ is self-dual, by  \eqref{factor22}, Corollary \ref{self-dual2}, and the classical zero-free region for $L(s, \chi)$, we  deduce that $L(s,\pi\times  {\check{\pi}}\times\chi)$ has  either no zeros or possibly only one simple real zero $\beta_\chi$ in the region
\begin{equation}
\label{cpi}
\sigma \geq 1-\frac{c_{\pi}}{\log{\left( ({\rm N} \mathfrak{q})(|t|+3)\right)} }
\end{equation}
for some $c_{\pi}>0$ only depending on $\pi$. More precisely, this region is obtained by the intersection of the zero-free region for $L(s, \Ad(\pi)\otimes \chi)$ given by Corollary \ref{self-dual2} with the classical zero-free region for $L(s, \chi)$ given in Corollary \ref{pi}. Note that since $\chi$ has degree one, then the constant $c$ in the classical zero-free region 
for $L(s, \chi)$ is absolute, so $c_\pi$ in \eqref{cpi} is independent of $\chi$. 
Thus, (iii) holds.

Now, as $\Ad(\pi)\otimes\chi$ is cuspidal and of degree 3,  the first part  of Theorem \ref{Siegelzero} yields the non-existence of the exceptional zero for $L(s,\Ad(\pi)\otimes\chi)$. Thus, if the exceptional zero $\beta_\chi$ of $L(s, \pi\times  \check{\pi}\times\chi)$ exists, it has to come from $L(s, \chi)$. Therefore, $\beta_\chi$ depends only on $\chi$.
In this case,
by Theorem \ref{Grossen},
there is a constant $\kappa (\epsilon)$, depending only on $\epsilon$, such that
$$
\beta_\chi \leq 1-\frac{\kappa(\epsilon)}{{\rm N} \mathfrak{q}^\epsilon},
$$
where $\mathfrak{q}$ is the modulus of the ray class character $\chi$. Thus, (iv) holds trivially.
\smallskip

(c) Under the assumptions, we know that $\pi\otimes \chi$ and $\pi^\prime$ are non-dihedral cuspidal representations of ${\rm GL}_2(\mathbb{A}_F)$ that are not twist-equivalent. Thus, it follows from Theorem \ref{GJ}(i) that $L(s, \pi\times\pi^\prime\times\chi)=L(s, (\pi \otimes \chi) \boxtimes \pi^\prime)$ for the cuspidal representation $(\pi \otimes \chi) \boxtimes \pi^\prime$ of ${\rm GL}_4(\mathbb{A}_F)$. So, $L(s, \pi\times \pi^\prime\times \chi)$ is entire.
This settles (ii).

The condition (iii) is a direct consequence of Corollary \ref{gl2}.


Finally, by Theorem \ref{Siegelzero}(ii), we know that under the given conditions 
$L(s,\pi\times \pi'\times \chi)$ has no exceptional zero. Thus, (iv) holds.

\smallskip
Hence, by Theorem \ref{main}, for $\Pi$ satisfying either (a), (b), or (c),   we have that for any $\gamma>0$, there exists  $c=c(\Pi)>0$ such that for any ideal $\mathfrak{q}$, with ${\rm N}{\mathfrak{q}}\leq (\log{x})^\gamma$, and any ideal $\mathfrak{a}$ relatively prime to $\mathfrak{q}$, \eqref{second} holds. Now \eqref{twotwo} follows from \eqref{second} and Proposition \ref{Wu-Ye}.

\section{Proof of Theorem \ref{main}}
\label{four}

In this section, we prove Theorem \ref{main}. 
We start by collecting some analytic properties of $L(s, \pi\times \pi^\prime\times \chi)$ in the following lemma. 
\begin{lemma}\label{Majid}
(i) For $\delta>0$, let $\mathbb{C}(\delta)$ denote the set
$$\mathbb{C}\backslash\{s\in\mathbb{C}:~
|s+\kappa_{\pi\times \pi^\prime\times \chi}(i,j,v)+2k|\leq\delta,\ \text{for } v\in S_\infty,\ 1\leq i\leq m,\ 1\leq j\leq m^\prime,\ \text{and integers } k\geq0\}.$$
Let $\sigma\leq -1/2$. Then for all $s=\sigma+it\in\mathbb{C}(\delta)$,
$$\dfrac{L'(s, \pi\times \pi^\prime\times \chi)}{L(s, \pi\times \pi^\prime\times\chi)}\ll_{\pi, \pi^\prime, \delta}\log{\left(\left({\rm N} \mathfrak{q} \right)|s|\right)}.$$
(ii) For any integer $m\geq2$, there is $T_m$, with $m\leq T_m\leq m+1$, such that
$$\dfrac{L'(\sigma\pm iT_m, \pi\times\pi^\prime\times \chi)}{L(\sigma\pm iT_m, \pi\times \pi^\prime\times \chi)}\ll_{\pi, \pi^\prime}
\log^2{\left(\left({\rm N} \mathfrak{q} \right) T_m\right)}$$
uniformly for $-2\leq\sigma\leq2$.\\
(iii) Let $N(t, \pi\times\pi^\prime\times\chi)$  
be the number of the zeros $\rho$ of $L(s, \pi\times \pi^\prime\times \chi)$ 
in the region $0\leq\Re(s)\leq1$, 
where 
$t-1\leq \Im(\rho) \leq t+1$. Then
$$N(t, \pi\times \pi^\prime\times \chi)\ll_{\pi, \pi^\prime}\log{\left(({\rm N}\mathfrak{q})(|t|+3)\right)}.$$
(iv) Let $$b_{\pi, \pi^\prime}(\chi)= \lim_{s\rightarrow 0} \left(\dfrac{L'(s, \pi\times\pi^\prime\times\chi)}{L(s, \pi\times\pi^\prime\times\chi)}-\dfrac{r}{s} \right),$$
where the 
integer $r\geq 0$ is the order of vanishing of $L(s, \pi\times\pi^\prime\times\chi)$ at $s=0$.
Then $$b_{\pi, \pi^\prime}(\chi)=O_{\pi, \pi^\prime}(\log{{\rm N} \mathfrak{q}})
-\sum_
{\substack{{0< |\Re(\rho)|\leq 1}\\ { |\Im(\rho)|\leq 1}}} \frac{1}{\rho},$$
where $\rho\neq 0$ ranges over the  non-trivial zeros of $L(s, \pi \times \pi^\prime\times \chi)$.
\end{lemma}
\begin{proof}
For (i) see \cite[p. 177]{Mor} for a single ${\rm GL}_2$ automorphic $L$-function over $F=\mathbb{Q}$. The general case is similar. See \cite[ Lemma 4.3(a)(d)]{LY} for (ii) and (iii) for $F=\mathbb{Q}$, again the proof for general $F$ is similar. 
The assertion (iv) is a consequence \cite[Proposition 5.7(2)]{IK}, \eqref{Bushnell-2}, and \eqref{infinity-twist}.
\end{proof}

The following lemma evaluates a contour integral that will appear in the proof.

\begin{lemma}
\label{integral}
Let $b>1$, $2\leq T\leq x$, and ${\rm N} \mathfrak{q} \leq x$. We have
\begin{align}\label{f_expa_0}
 \begin{split} 
&\frac{1}{2\pi i} \int_{b-iT}^{b+iT}  -\frac{L^\prime}{L}(s, \pi\times \pi^\prime\times \chi) \frac{x^s}{s} ds 
=\delta(\pi \times \pi^\prime\times \chi) x
- \sum_{\substack{{0< |\Re(\rho)|\leq 1}\\ {|\Im(\rho)|\leq T}}}  \frac{x^\rho}{\rho}\\
&- b_{\pi, \pi^\prime}(\chi)+O_{\pi, \pi^\prime} \left( \frac{x\log{x}}{T} \right)+O_{\pi, \pi^\prime} \left(x^{ {1-\frac{1}{m^2+1}-\frac{1}{(m^\prime)^2+1}}}  \right) ,
 \end{split}
\end{align}
where $\delta(\pi\times\pi^\prime\times \chi)$ is $1$ if $L(s, \pi\times\pi^\prime\times\chi)$ has a simple pole at $s=1$ and is zero otherwise,
$\rho\neq 0$ ranges over the  non-trivial zeros of $L(s, \pi \times \pi^\prime\times \chi)$,
and $b_{\pi, \pi^\prime}(\chi)$ is the expression defined in part (iv) of Lemma \ref{Majid}. 
\end{lemma}
\begin{proof}
The proof is standard and follows closely the arguments given in \cite[Chapter 19]{D} for the case $\pi=\pi^\prime=1$ over $F=\mathbb{Q}$  and the arguments given in Proposition 4.2 of \cite{ANM} for a single automorphic $L$-function $\pi$ over $F=\mathbb{Q}$. In fact \eqref{f_expa_0} is a consequence of computing the residues of the integrand upon moving the line of integration to the left {and employing}  \eqref{infinite} and parts (i), (ii), and (iii) of Lemma \ref{Majid}.  Note that the errors terms 
$$O(\log{x})+O_{\pi, \pi^\prime} \left( \frac{x\log^2(({\rm N} \mathfrak{q}) T)}{T\log{x}} \right)+O_{\pi, \pi^\prime} \left( \frac{x\log(({\rm N} \mathfrak{q}) T)}{T} \right)$$
appearing in the process can be combined as the first error term in \eqref{f_expa_0} under the assumptions $T\leq x$ and ${\rm N} \mathfrak{q} \leq x$. See \cite[Proposition 4.2]{ANM} for details.
\end{proof}

We also need a version of the truncated Perron's formula due to Liu and Ye  \cite[Theorem 2.1]{LY07}.

\begin{lemma}
\label{LYe}
Let $f(s)=\sum_{n=1}^{\infty} \frac{a_n}{n^s}$ be an absolutely convergent series in the half-plane $\sigma >\sigma_a$.  Let $B(\sigma)=\sum_{n=1}^{\infty} \frac{|a_n|}{n^\sigma}$ for $\sigma>\sigma_a$. Then for $b>\sigma_a$, $x\geq 2$, $T\geq 2$, and $H\geq 2,$
$$\sum_{n\leq x} a_n=\frac{1}{2\pi i} \int_{b-iT}^{b+iT} f(s) \frac{x^s}{s}ds+ O \left( \sum_{x-x/H<n\leq x+x/H} |a_n| \right) +O\left(\frac{Hx^bB(b)}{T} \right).$$
\end{lemma}

We now have all the necessary tools for the proof in our disposal.

\begin{proof}[Proof of Theorem \ref{main}]
We only describe the proof that implies \eqref{second}. The argument can be adjusted to obtain \eqref{first-2}. 
Let $\epsilon_0>0$ be as given in the assumption (i) in the statement of the theorem.
Assume that $x\geq 2$, $T\geq 4$,  $T\leq x^{\epsilon_0}$, and ${\rm N} \mathfrak{q} \leq x$.
In Lemma \ref{LYe}, set $H={T^{1/2}}$, $b=1+1/\log{x}$, and $f(s)=-\frac{L'}{L}(s, \pi\times\pi^\prime\times\chi)$. 
Then employing Proposition \ref{Landau} for $u=x/T^{1/2}$ and the bound \eqref{bound-general}
for $\mathfrak{p}\mid (A_\pi, A_{\pi^\prime})$ yield
\begin{align}\label{perron}
 \begin{split}
\sideset{}{^*}\sum_{{\rm N}\mathfrak{n}\leq x} \Lambda(\mathfrak{n}) a_{\pi\times \pi^\prime}(\mathfrak{n}) \chi(\mathfrak{n})
&=\ \frac{1}{2\pi i} \int_{b-iT}^{b+iT}  -\frac{L^\prime}{L}(s, \pi\times \pi^\prime\times \chi) \frac{x^s}{s} ds\\
&+O_{\pi, \pi^\prime} \left(\frac{x\log{x}}{T^{1/2}} \right)+O_{\pi, \pi^\prime} \left( x^{ {1-\frac{1}{m^2+1}-\frac{1}{(m^\prime)^2+1}}} (\log{x})^2\right).
 \end{split}
\end{align}

For the integral in \eqref{perron}, by employing Lemma \ref{integral} and part (iv) of Lemma \ref{Majid},
we deduce 

\begin{align}\label{f_expa}
 \begin{split} 
&\frac{1}{2\pi i} \int_{b-iT}^{b+iT}  -\frac{L^\prime}{L}(s, \pi\times \pi^\prime\times \chi) \frac{x^s}{s} ds 
=\delta(\pi \times \pi^\prime\times \chi) x
- \sum_{\substack{{0< |\Re(\rho)|\leq 1}\\ { |\Im(\rho)|\leq T}
}}  \frac{x^\rho}{\rho}\\
&+\sum_{
\substack{{0< |\Re(\rho)|\leq 1}\\ 
{ |\Im(\rho)|\leq 1}\\{\rho\neq \beta_\chi, 1-\beta_\chi}
}} \frac{1}{\rho}+\frac{1}{\beta_\chi}+\frac{1}{1-\beta_\chi}+ O_{\pi, \pi^\prime} \left( \frac{x\log{x}}{T} \right)+O_{\pi, \pi^\prime} \left( x^{ {1-\frac{1}{m^2+1}-\frac{1}{(m^\prime)^2+1}}} \right),
 \end{split}
\end{align}
where $\beta_\chi$ is the possible exceptional zero of $L(s, \pi\times\pi^\prime\times\chi)$, and, as later, all terms contributed by $\beta_{\chi}$ should be omitted if $\beta_\chi$ does not exist.  

Next we focus on the sums in \eqref{f_expa} involving the non-trivial zeros $\rho$. First of all, by the assumption (iii) and the symmetry of the non-trivial zeros respect to line $\Re(s)=1/2$, for any low-lying zero $\rho\neq \beta_\chi, 1-\beta_\chi$, we have $\rho^{-1}=O\left(\log{{\rm N}(\mathfrak{q})}\right)$. Thus, by Lemma \ref{Majid}(iii), we deduce that   
\begin{equation}
\label{low}
\sum_{
\substack{{0< |\Re(\rho)|\leq 1}\\ 
{ |\Im(\rho)|\leq 1}\\{\rho \neq \beta_\chi, 1-\beta_\chi}
}} \frac{1}{\rho} =O_{\pi, \pi^\prime}(\log^2{{\rm N} \mathfrak{q}}),
\end{equation}
so this term can be absorbed in the second error term in the right-hand side of \eqref{f_expa}.
Secondly, by Lemma \ref{Majid}(iii), we have 
$$
\sum_{\substack{{0< |\Re(\rho)|\leq 1}\\ {3<|\Im(\rho)|\leq T}}
}\frac{1}{|\rho|}\ll\sum_{3\le t<T}\frac{N(t,\pi \times \pi^\prime\times \chi)}{t}\ll_{\pi, \pi^\prime} (\log T)\log{(({\rm N} \mathfrak{q}) T)} .
$$
This together with the assumption  (iii), the classical zero-free region of $L(s,\pi \times \pi^\prime\times \chi)$,  gives
\begin{equation}
\label{zero-T}
\sum_{\substack{{0< |\Re(\rho)|\leq 1}\\ {3<|\Im(\rho)|\leq T}
}}\Big|\frac{x^{\rho}}{\rho}\Big|\ll_{\pi, \pi^\prime} 
(\log T) (\log( ({\rm N} \mathfrak{q}) T))
x^{1-c_{\pi , \pi'}(\log(({\rm N}{\mathfrak{q}})(T+3))
)^{-1}}.
\end{equation}

Similarly, for non-exceptional zeros $\rho$ with $|\Im(\rho)|\leq 3$, the assumption (iii)
yields 
\begin{equation}
\label{zero-3}
\sum_{\substack{{0<|Re(\rho)|\leq 1}\\{|\Im(\rho)|\leq 3}
\\{ \rho\neq \beta_\chi, 1-\beta_\chi}}}\Big|\frac{x^{\rho}}{\rho}\Big|\ll_{\pi, \pi^\prime} {{(\log^2{{\rm N}{\mathfrak{q}}})}} x^{1-c_{\pi , \pi'}(\log( 6 {\rm N}{\mathfrak{q}} ))^{-1}},
\end{equation}
where $\beta_\chi$ is the possible exceptional zero of $L(s,\pi \times \pi^\prime\times \chi)$.
Now let $T=\exp( (\log{x})^{1/2})$. Then, for ${\rm N} \mathfrak{q} \leq  \exp((\log{x})^{1/2})$, \eqref{zero-T} and \eqref{zero-3} yield
\begin{equation}
\label{zero-3T}
\sum_{\substack{{0<|Re(\rho)|\leq 1}\\{|\Im(\rho)|\leq T}
\\{ \rho\neq \beta_\chi, 1-\beta_\chi}}}\Big|\frac{x^{\rho}}{\rho}\Big|\ll_{\pi, \pi^\prime} x \exp(-{\hat{c}}_{\pi, \pi^\prime} (\log{x})^{1/2})
\end{equation}
for a constant $\hat{c}_{\pi, \pi^\prime}$ depending only on $\pi$ and $\pi^\prime$.

Recall that $\chi$ is of modulus $\mathfrak{q}$ and that by the assumption (iv), for any $\epsilon>0$, there is a constant $\kappa(\epsilon):=\kappa(\epsilon, \pi, \pi^\prime)$, so that the exceptional zero $\beta_\chi$ of $L(s,\pi\times\pi^\prime\times \chi)$ satisfies $\beta_\chi\le  1-\kappa(\epsilon){\rm N} \mathfrak{q}^{-\epsilon}$. Thus, for 
${\rm N} \mathfrak{q}\le (\log x)^{\gamma}$, we have
\begin{equation}
\label{S-zero}
\frac{x^{\beta_\chi}-1}{\beta_\chi}+\frac{x^{1-\beta_\chi}-1}{1-\beta_\chi}\le 2 \frac{x^{\beta_\chi}-1}{1-\beta_\chi}\ll
 \frac{x^{1- \kappa(\epsilon)(\log x)^{-\epsilon\gamma}}}{\kappa(\epsilon)(\log x)^{-\epsilon\gamma}}.
\end{equation}

For $\gamma>0$,
let $\epsilon=1/3\gamma$. Then, for $x$ satisfying $ (\log{x})^\gamma \leq  \exp((\log{x})^{1/2})$, $T=\exp((\log{x})^{1/2})$, and ${\rm N} \mathfrak{q} \leq (\log{x})^\gamma$, from \eqref{zero-3T} and \eqref{S-zero}, we get
\begin{align}
\label{S-rho}
 \begin{split}
\sum_{\substack{{0< |\Re(\rho)|\leq 1}\\ {|\Im(\rho)|\leq T}\\{\rho \neq \beta_\chi, 1-\beta_\chi}
}}\Big|\frac{x^{\rho}}{\rho}\Big| &+ \left| \frac{x^{\beta_\chi}-1}{\beta_\chi}+\frac{x^{1-\beta_\chi}-1}{1-\beta_\chi}\right|\\
&\ll_{\pi, \pi^\prime} 
x \left( \exp(-{\hat{c}}_{\pi, \pi^\prime} (\log{x})^{1/2})+  
\kappa(1/3\gamma)^{-1} (\log{x})^{1/3} \exp(-\kappa(1/3\gamma) (\log{x})^{2/3})
\right).
  \end{split}
\end{align}
%

Let $C:=C(\pi, \pi^\prime, \gamma)$ be a positive constant such that for $x>C$, 
\begin{equation}
\label{T-choice}
\max\{4, (\log{x})^{\gamma}\} \leq T= \exp((\log x)^{1/2}) \leq x^{\epsilon_0}.
\end{equation}
Thus, for $x>C$, $T$ as in \eqref{T-choice}, and ${\rm N} \mathfrak{q}\le (\log x)^{\gamma}$, by employing \eqref{S-rho} in \eqref{f_expa}, the asymptotic formula \eqref{perron} can be written as
\begin{equation}
\label{final}
\sideset{}{^*}\sum_{{\rm N}\mathfrak{n}\leq x} \Lambda(\mathfrak{n}) a_{\pi\times \pi^\prime}(\mathfrak{n}) \chi(\mathfrak{n})=\delta(\pi \times \pi^\prime\times \chi) x+O_{\pi, \pi^\prime, \gamma}\left(x\exp\left(-{c}(\log{x})^{1/2}\right)\right)
\end{equation}
for some $c:=c(\pi, \pi^\prime)>0$.

The final result follows from \eqref{final}, the orthogonality property of ray class characters, i.e.,
\begin{equation}\label{DD}
\sideset{}{^*}\sum_{\substack {{\rm N} \mathfrak{n}\leq x\\ \mathfrak{n} \sim \mathfrak{a}~\mymod{ \mathfrak{q}} }} \Lambda(\mathfrak{n}) a_{\pi\times \pi^\prime}(\mathfrak{n})
=\frac{1}{h(\mathfrak{q})}\sum_{\chi\mymod{ \mathfrak{q}}}\overline{\chi}(\mathfrak{a})\sideset{}{^*}\sum_{{\rm N} \mathfrak{n}\leq x} \Lambda(\mathfrak{n}) a_{\pi \times \pi^\prime} (\mathfrak{n}) \chi(\mathfrak{n}),
\end{equation}
and the assumption (ii) in the statement of the theorem.
\end{proof}

\section{Theorem \ref{main2} implies Theorem \ref{third} }

\label{five}

It is enough to show that the conditions (i), (ii), (iii), (iv), and (v) of Theorem \ref{main2} hold for pairs $\pi$ and $\pi^\prime$ associated with $\Pi$ in Theorem \ref{third}. Then \eqref{last-main-est} together with Proposition \ref{Wu-Ye} imply \eqref{toto}. Note that (i), (ii), (iii) and (v) of Theorem \ref{main2} are the same as (i), (ii), (iii) and (iv) in Theorem \ref{main}. So following the arguments of Section \ref{Section3}, the conditions (i), (ii), (iii), and (v) hold for $\pi$ and $\pi^\prime$ associated with $\Pi$ in Theorem \ref{third}. Since $\Pi$ in (a) is cuspidal and, by Theorem \ref{GJ}(i),  in (b) ({resp.,} (c)) is automorphic ({resp.,} cuspidal), then, by Theorem \ref{log-free},  the condition (iv) of Theorem \ref{main2} also holds for the corresponding $\pi$ and $\pi^\prime$.

\section{Proof of Theorem \ref{main2}}
\label{six}
\begin{proof}[Proof of Theorem \ref{main2}]
We only describe the proof that implies \eqref{last-main-est} since the argument can be adjusted to establish \eqref{fourfour}. The proof closely follows the proof of Theorem 1 in \cite{Moto}.
Let $\epsilon_0>0$ be as 
given in the assumption (i) of the theorem and without loss of generality assume that $0<\epsilon_0\leq 4/5$. Let $x$, $T$, ${\rm N}\mathfrak{q}$, $H$, $b$, and $f(s)$ be as in the beginning of the proof of Theorem \ref{main}, so by following the initial steps of the proof of Theorem \ref{main} 
we get \eqref{perron} and \eqref{f_expa}. Let  $0<\theta\leq\epsilon_0/4$, $T=x^{4\theta}$,
$y=x^{1-\theta}$, and ${\rm N} \mathfrak{q} \leq x^\theta$.  Then, from \eqref{perron}, \eqref{f_expa}, and \eqref{low}, we deduce
\begin{equation}
\label{first-Motohashi}
\sideset{}{^*}\sum_{x-y< {\rm N} \mathfrak{n}\leq x }  \Lambda(\mathfrak{n})   a_{\pi\times \pi'}(\mathfrak{n})\chi(\mathfrak{n}) 
= \delta(\pi\times\pi'\times \chi)y -\sum_{\substack{{0<|\Re(\rho)|\leq 1}\\{|\Im(\rho)|\leq T}}} \hat{g}(\rho)  +O_{\pi,\pi'}(yx^{-\theta}\log x), 
\end{equation}
where
$\hat{g}$ is the Mellin transform of $g=\textbf{1}_{(x-y, x]}$ (the indicator function of the interval $(x-y, x]$).
Next, 
we note that for $T\geq 4$ we can find a constant $\tilde{c}_{\pi, \pi^\prime}$ such that $$\frac{{c}_{\pi,\pi'}}{\log  \left(({\rm N}\mathfrak{q})(T+3)\right)}\geq  \frac{\tilde{c}_{\pi,\pi'}}{\log \left(({\rm N}\mathfrak{q}) T\right)},$$
where $c_{\pi, \pi^\prime}$ is the constant given in the assumption (iii).

To control the zero-sum in \eqref{first-Motohashi}, we shall apply the assumptions (iii) {and} (iv) of the theorem, and \cite[Theorem 5.8]{IK} together with \eqref{Bushnell-2},
to deduce
\begin{align*}
\sum_{\rho\neq \beta_\chi} |\hat{g}(\rho)|
& \le \int_{x-y}^{x}  (\log t)\left( \int_0^{1-\frac{\tilde{c}_{\pi,\pi'}}{\log \left(({\rm N}\mathfrak{q}) T\right)}}  N(\sigma,T, \pi\times\pi'\times \chi) t^{\sigma-1}  d \sigma\right) dt\\
& + N(0,T, \pi\times\pi'\times \chi)\int_{x-y}^{x}  \frac{dt}{t} \\
&\ll_{\pi,\pi'} y (\log x) \int_{0}^{1-\frac{\tilde{c}_{\pi,\pi'}}{\log \left(({\rm N} \mathfrak{q}) T\right)}} (  ({\rm N}\mathfrak{q}) T)^{d_{\pi\times \pi^\prime}(1-\sigma)} x^{\sigma-1}  d \sigma+  (y/x) T\log \left(({\rm N}\mathfrak{q})  T\right),
\end{align*}
where $\beta_\chi$ is the possible exceptional zero of $L(s,\pi\times\pi'\times \chi)$. Now let  $\theta \leq 1/(10 d_{\pi, \pi^\prime})$ so that $\left( ({\rm N} \mathfrak{q})T\right) ^{d_{\pi, \pi^\prime}} \leq x^{1/2}$. (Recall that ${\rm N} \mathfrak{q} \leq x^\theta$ and $T=x^{4\theta}$.) With these choices, we get
\begin{equation}
\label{gbar-bound}
\sum_{\rho\neq \beta_\chi  } |\hat{g}(\rho)| \ll_{\pi, \pi^\prime} y \exp \left(-\frac{c_0}{\theta} \right)+y{x^{4\theta-1}\log{x}}
\end{equation}
for a constant $c_0$ depending on $\pi$ and $\pi^\prime$. Now, by choosing $x$ such that $(\log{x})^{-1/2} \leq \theta$ we deduce $x^{-\theta} \log{x} < \exp(-1/2\theta)$. So, by adjusting $\tilde{c}_{\pi, \pi^\prime}$, we see that the error terms $O_{\pi, \pi^\prime}(y x^{-\theta} \log{x})$ in \eqref{first-Motohashi} and $O_{\pi, \pi^\prime}(y x^{{4\theta}-1} \log{x})$ in \eqref{gbar-bound} can be absorbed in the term $y \exp \left(-\frac{c_0}{\theta} \right)$ in \eqref{gbar-bound}. (Note that $x^{4\theta-1} \leq x^{-\theta}$ since $\theta\leq \epsilon_0/4\leq 1/5$.)
Thus, inserting the derived bound for  $\sum_{\rho\neq \beta_\chi} |\hat{g}(\rho)|$
in \eqref{first-Motohashi} yields   
\begin{equation}
\label{second-Motohashi}
\sideset{}{^*}\sum_{x-y< {\rm N} \mathfrak{n}\leq x }  \Lambda(\mathfrak{n})  a_{\pi\times \pi'}(\mathfrak{n})\chi(\mathfrak{n}) 
= \delta(\pi\times\pi'\times \chi)y +O( |\hat{g}(\beta_\chi)|)+O\left(y \exp \left(-\frac{c_0}{\theta} \right) \right).   
\end{equation}

To treat the term involving $|\hat{g}(\beta_\chi)|$, observe that the mean value theorem implies that
$$|\hat{g}(\beta_\chi) | = \int_{x-y}^{x} t^{\beta_\chi -1} dt=\frac{x^{\beta_\chi}-(x-y)^{\beta_\chi}}{\beta_\chi}= y\xi^{\beta_\chi-1},$$
for some $\xi\in (x-y, x)$. 
We conclude that
$$|\hat{g} (\beta_\chi)| =O(y x^{\beta_\chi -1 }).$$

Now inserting this bound in \eqref{second-Motohashi} and employing the orthogonality relation  \eqref{DD}, together with the assumptions (ii) and (v) in the statement of the theorem,
imply the result.
\end{proof}


\begin{rezabib}

\bib{ANM}{article}{
   author={Akbary, Amir},
   author={Ng, Nathan},
   author={Shahabi, Majid},
   title={Limiting distributions of the classical error terms of prime
   number theory},
   journal={Q. J. Math.},
   volume={65},
   date={2014},
   number={3},
   pages={743--780},
   issn={0033-5606},
   review={\MR{3261965}},
   doi={10.1093/qmath/hat059},
}
	
\bib{AT}{article}{
   author={Akbary, Amir},
   author={Trudgian, Timothy S.},
   title={A log-free zero-density estimate and small gaps in coefficients of
   $L$-functions},
   journal={Int. Math. Res. Not. IMRN},
   date={2015},
   number={12},
   pages={4242--4268},
   issn={1073-7928},
   review={\MR{3356752}},
   doi={10.1093/imrn/rnu065},
}

\bib{Ba97}{article}{
   author={Banks, William D.},
   title={Twisted symmetric-square $L$-functions and the nonexistence of
   Siegel zeros on ${\rm GL}(3)$},
   journal={Duke Math. J.},
   volume={87},
   date={1997},
   number={2},
   pages={343--353},
   issn={0012-7094},
   review={\MR{1443531}},
   doi={10.1215/S0012-7094-97-08713-5},
}

%
\bib{B}{article}{
   author={Brumley, Farrell},
   title={Effective multiplicity one on ${\rm GL}_N$ and narrow zero-free
   regions for Rankin-Selberg $L$-functions},
   journal={Amer. J. Math.},
   volume={128},
   date={2006},
   number={6},
   pages={1455--1474},
   issn={0002-9327},
   review={\MR{2275908}},
}

%
%



\bib{D}{book}{
   author={Davenport, Harold},
   title={Multiplicative number theory},
   series={Graduate Texts in Mathematics},
   volume={74},
   edition={3},
   note={Revised and with a preface by Hugh L. Montgomery},
   publisher={Springer-Verlag, New York},
   date={2000},
   pages={xiv+177},
   isbn={0-387-95097-4},
   review={\MR{1790423}},
}

\bib{GJ}{article}{
   author={Gelbart, Stephen},
   author={Jacquet, Herv\'{e}},
   title={A relation between automorphic representations of ${\rm GL}(2)$
   and ${\rm GL}(3)$},
   journal={Ann. Sci. \'{E}cole Norm. Sup. (4)},
   volume={11},
   date={1978},
   number={4},
   pages={471--542},
   issn={0012-9593},
   review={\MR{533066}},
}

	
\bib{G}{article}{
   author={Goldstein, Larry Joel},
   title={A generalization of the Siegel-Walfisz theorem},
   journal={Trans. Amer. Math. Soc.},
   volume={149},
   date={1970},
   pages={417--429},
   issn={0002-9947},
   review={\MR{274416}},
   doi={10.2307/1995404},
}



\bib{HR}{article}{
   author={Hoffstein, Jeffrey},
   author={Ramakrishnan, Dinakar},
   title={Siegel zeros and cusp forms},
   journal={Internat. Math. Res. Notices},
   date={1995},
   number={6},
   pages={279--308},
   issn={1073-7928},
   review={\MR{1344349}},
   doi={10.1155/S1073792895000225},
}

\bib{HB}{article}{
   author={Humphries, Peter},
   author={Brumley, Farrell},
   title={Standard zero-free regions for Rankin-Selberg $L$-functions via
   sieve theory},
   journal={Math. Z.},
   volume={292},
   date={2019},
   number={3-4},
   pages={1105--1122},
   issn={0025-5874},
   review={\MR{3980284}},
   doi={10.1007/s00209-018-2136-8},
}
		
\bib{HT}{article}{
   author={Humphries, Peter},
   author={Thorner, Jesse},
   title={Zeros of Rankin-Selberg $L$-functions in families},
   journal={arXiv:2103.05634 [math.NT]},
   date={2021},
   pages={54 pages},
}

\bib{I00}{article}{
   author={Ichihara, Yumiko},
   title={The Siegel-Walfisz theorem for Rankin-Selberg $L$-functions
   associated with two cusp forms},
   journal={Acta Arith.},
   volume={92},
   date={2000},
   number={3},
   pages={215--227},
   issn={0065-1036},
   review={\MR{1752026}},
   doi={10.4064/aa-92-3-215-227},
}

\bib{IK}{book}{
   author={Iwaniec, Henryk},
   author={Kowalski, Emmanuel},
   title={Analytic number theory},
   series={American Mathematical Society Colloquium Publications},
   volume={53},
   publisher={American Mathematical Society, Providence, RI},
   date={2004},
   pages={xii+615},
   isbn={0-8218-3633-1},
   review={\MR{2061214}},
   doi={10.1090/coll/053},
}

\bib{KS}{article}{
   author={Kim, Henry H.},
   author={Shahidi, Freydoon},
   title={Functorial products for ${\rm GL}_2\times{\rm GL}_3$ and the
   symmetric cube for ${\rm GL}_2$},
   note={With an appendix by Colin J. Bushnell and Guy Henniart},
   journal={Ann. of Math. (2)},
   volume={155},
   date={2002},
   number={3},
   pages={837--893},
   issn={0003-486X},
   review={\MR{1923967}},
   doi={10.2307/3062134},
}

\bib{LO-T19}{article}{
   author={Lemke Oliver, Robert J.},
   author={Thorner, Jesse},
   title={Effective log-free zero density estimates for automorphic
   $L$-functions and the Sato-Tate conjecture},
   journal={Int. Math. Res. Not. IMRN},
   date={2019},
   number={22},
   pages={6988--7036},
   issn={1073-7928},
   review={\MR{4032182}},
   doi={10.1093/imrn/rnx309},
}

\bib{LY}{article}{
   author={Liu, Jianya},
   author={Ye, Yangbo},
   title={Superposition of zeros of distinct $L$-functions},
   journal={Forum Math.},
   volume={14},
   date={2002},
   number={3},
   pages={419--455},
   issn={0933-7741},
   review={\MR{1899293}},
   doi={10.1515/form.2002.020},
}

\bib{LY07}{article}{
   author={Liu, Jianya},
   author={Ye, Yangbo},
   title={Perron's formula and the prime number theorem for automorphic
   $L$-functions},
   journal={Pure Appl. Math. Q.},
   volume={3},
   date={2007},
   number={2, Special Issue: In honor of Leon Simon.},
   pages={481--497},
   issn={1558-8599},
   review={\MR{2340051}},
   doi={10.4310/PAMQ.2007.v3.n2.a4},
}

%
%


\bib{Mi06}{article}{
   author={Michel, Philippe},
   title={Analytic number theory and families of automorphic $L$-functions},
   conference={
      title={Automorphic forms and applications},
   },
   book={
      series={IAS/Park City Math. Ser.},
      volume={12},
      publisher={Amer. Math. Soc., Providence, RI},
   },
   date={2007},
   pages={181--295},
   review={\MR{2331346}},
   doi={10.1090/pcms/012/05},
}

\bib{Mi56}{article}{
   author={Mitsui, Takayoshi},
   title={Generalized prime number theorem},
   journal={Jap. J. Math.},
   volume={26},
   date={1956},
   pages={1--42},
   review={\MR{0092814}},
   doi={10.4099/jjm1924.26.0-1},
}


\bib{Mo}{article}{
   author={Moreno, Carlos Julio},
   title={The Hoheisel phenomenon for generalized Dirichlet series},
   journal={Proc. Amer. Math. Soc.},
   volume={40},
   date={1973},
   pages={47--51},
   issn={0002-9939},
   review={\MR{327682}},
   doi={10.2307/2038629},
}

\bib{Mor}{article}{
   author={Moreno, C. J.},
   title={Explicit formulas in the theory of automorphic forms},
   conference={
      title={Number Theory Day (Proc. Conf., Rockefeller Univ., New York,
      1976)},
   },
   book={
      publisher={Springer, Berlin},
   },
   date={1977},
   pages={73--216. Lecture Notes in Math., Vol. 626},
   review={\MR{0476650}},
}

\bib{Moto}{article}{
   author={Motohashi, Yoichi},
   title={On sums of Hecke-Maass eigenvalues squared over primes in short
   intervals},
   journal={J. Lond. Math. Soc. (2)},
   volume={91},
   date={2015},
   number={2},
   pages={367--382},
   issn={0024-6107},
   review={\MR{3355106}},
   doi={10.1112/jlms/jdu079},
}

\bib{P82}{article}{
   author={Perelli, Alberto},
   title={On the prime number theorem for the coefficients of certain
   modular forms},
   conference={
      title={Elementary and analytic theory of numbers},
      address={Warsaw},
      date={1982},
   },
   book={
      series={Banach Center Publ.},
      volume={17},
      publisher={PWN, Warsaw},
   },
   date={1985},
   pages={405--410},
   review={\MR{840485}},
}

\bib{R}{article}{
   author={Ramakrishnan, Dinakar},
   title={Modularity of the Rankin-Selberg $L$-series, and multiplicity one
   for ${\rm SL}(2)$},
   journal={Ann. of Math. (2)},
   volume={152},
   date={2000},
   number={1},
   pages={45--111},
   issn={0003-486X},
   review={\MR{1792292}},
   doi={10.2307/2661379},
}

\bib{RW03}{article}{
   author={Ramakrishnan, D.},
   author={Wang, S.},
   title={On the exceptional zeros of Rankin-Selberg $L$-functions},
   journal={Compositio Math.},
   volume={135},
   date={2003},
   number={2},
   pages={211--244},
   issn={0010-437X},
   review={\MR{1955318}},
   doi={10.1023/A:1021761421232},
}

\bib{ST}{article}{
   author={Soundararajan, Kannan},
   author={Thorner, Jesse},
   title={Weak subconvexity without a Ramanujan hypothesis},
   note={With an appendix by Farrell Brumley},
   journal={Duke Math. J.},
   volume={168},
   date={2019},
   number={7},
   pages={1231--1268},
   issn={0012-7094},
   review={\MR{3953433}},
   doi={10.1215/00127094-2018-0065},
}

\bib{T}{article}{
   author={Timofeev, N. M.},
   title={Distribution of arithmetic functions in short intervals in the
   mean with respect to progressions},
   language={Russian},
   journal={Izv. Akad. Nauk SSSR Ser. Mat.},
   volume={51},
   date={1987},
   number={2},
   pages={341--362, 447},
   issn={0373-2436},
   translation={
      journal={Math. USSR-Izv.},
      volume={30},
      date={1988},
      number={2},
      pages={315--335},
      issn={0025-5726},
   },
   review={\MR{897001}},
   doi={10.1070/IM1988v030n02ABEH001013},
}

\bib{PJ19}{article}{
   author={Wong, P.-J.},
   title={Bombieri-Vinogradov theorems for modular forms and applications},
   journal={Mathematika},
   volume={66},
   date={2020},
   pages={200--229},
   doi={10.1112/mtk.12014},
}

\bib{WY07}{article}{
   author={Wu, Jie},
   author={Ye, Yangbo},
   title={Hypothesis H and the prime number theorem for automorphic
   representations},
   journal={Funct. Approx. Comment. Math.},
   volume={37},
   date={2007},
   pages={461--471},
   issn={0208-6573},
   review={\MR{2364718}},
   doi={10.7169/facm/1229619665},
}

\end{rezabib}

\end{document}